\documentclass{amsart}
\usepackage{amsthm, amsmath, amssymb, latexsym, amsbsy,hyperref}
\title{Obstructions to Genericity in Study of Parametric Problems in Control Theory}

\author{Viktor Levandovskyy and Eva Zerz}
\address{Lehrstuhl D f\"ur Mathematik, RWTH Aachen, Templergraben 64, D-52062 Aachen, Germany}
\email{[viktor.levandovskyy,eva.zerz]@math.rwth-aachen.de}

\theoremstyle{definition}
\newtheorem{theorem}{Theorem}[section]
\newtheorem{proposition}[theorem]{Proposition}
\newtheorem{lemma}[theorem]{Lemma}

\newtheorem{algorithm}[theorem]{Algorithm}

\theoremstyle{definition}

\newtheorem{definition}[theorem]{Definition}
\newtheorem{example}[theorem]{Example}

\theoremstyle{definition}
\newtheorem{remark}[theorem]{Remark}

\keywords{generic properties, systems with parameters, Gr\"obner bases, parametric Gr\"obner bases, non-commutative Gr\"obner bases}

\DeclareMathOperator{\Ext}{Ext}
\DeclareMathOperator{\Hom}{Hom}
\DeclareMathOperator{\Mat}{Mat}
\DeclareMathOperator{\syz}{Syz}

\newcommand{\K}{{\mathbb K}}
\newcommand{\C}{{\mathbb C}}
\newcommand{\Q}{{\mathbb Q}}

\renewcommand{\d}{\partial}

\DeclareMathOperator{\lm}{lm}

\DeclareMathOperator{\krdim}{Kr.dim}

\begin{document}


\maketitle
\tableofcontents

\begin{flushright}
Dedicated to Michael S. Yakir
\end{flushright}
\section{Introduction}

{\bf Motivation:}
Problems from control theory often involve a set of physical
parameters, for instance, masses, spring constants, and damping coefficients
with mechanical systems,  or resistances, capacitances, and
inductances with electrical circuits.
The structural properties of the control system
may depend crucially on the specific choice of concrete 
parameter values. In many relevant
examples, a system is generically controllable (i.e.,
controllable for almost all possible parameter values),
but becomes uncontrollable when certain relations 
between the parameters are fulfilled. 

It has been shown for many system classes of practical 
interest that controllability amounts
to the torsion--freeness of a module associated to the system.
For example, if $A={\K}[\partial_1,\ldots,\partial_n]$ for a
field $\K$ and $\mathcal{F}=\mathcal{C}^{\infty}(\mathbb{R}^n,{\K})$, the system given by the linear
constant-coefficient partial differential equations 
$R(\partial_1,\ldots,\partial_n)w=0$, where $R\in A^{p\times q}$ and
$w\in \mathcal{F}^q$,
is controllable if and only if $M=A^{1\times q}/A^{1\times p}R$
is torsion-free. 
Then the \textbf{parametric controllability} problem can be formulated as follows:
Given $R\in A^{p\times q}$, where 
$A={\K}(p_1,\ldots,p_t)[\partial_1,\ldots,\partial_n]$
for some parameters $p_1,\ldots,p_t$,
find out whether $M$ is generically  torsion-free, and moreover,
determine the relations among the $p_i$ that will cause torsion
elements in $M$. In general, we pose this question for a left
module $M$ over a (non--commutative) algebra $A$. The antipode notion
of controllability is autonomy, which happens for systems
being torsion. For a corresponding system module $M$, this means
that $M$ is annihilated by a non-zero ideal in $A$.

To a system module $M$ over a system algebra $A$ one associates 
the \textbf{transposed module} $N = N(M)$, defined as follows. Let 
the left module $M$ be presented by a matrix $R \in A^{p \times q}$,
then $N(M)$ is a right module with the presentation 
matrix $R^T \in A^{q \times p}$.

Then, there is an alternative description of torsion--freeness (controllability) and torsion (autonomy) in the language of homological algebra. Namely, \\
$M$ is torsion--free if and only if $\Ext^1_A(N(M),A)=0$, and \\
$M$ is a torsion module if and only if $\Ext^0_A(M,A)= \Hom_A(M,A) = 0$.

For a survey of the correspondence between control systems and
their system modules, see \cite{LZ05} and the references therein,
in particular, the works of Oberst, Pommaret and Quadrat. 
A general approach to parametric modules including the case 
when the parameters are non--constant was introduced in \cite{PQ1997,PQ2000}.
The authors showed that so--called trees of integrability conditions,
depending on parameters of the system, determine the control--theoretic
properties of the system. These trees result into systems of partial
differential equations and nonlinear differential conditions.

The situation described above was our original motivation for
studying parameter-dependent questions of homological algebra
such as the specific problem outlined above. It turned out that 
apart from the concrete application area, it is a challenging 
task for computer algebra to investigate parametric modules, and in 
particular, to get a grip on the special values of parameters
that cause a qualitative change of structural module properties.
These questions reach far beyond the limited set of algebras
that typically arise in control theory. Roughly speaking,
the problem can be tackled from the computational point of view
for virtually every algebra that is accessible to Gr\"obner basis techniques.
The main idea is simple but effective: it consists in a careful monitoring
of denominators of cofactors that appear during Gr\"obner basis computation.
Thus in this article we continue with the investigations, started
in the articles \cite{LZ05, LZ06}.\\

{\bf Outline of the paper:}
In this paper, we give an algorithm for answering the
following \\

{\em Question from Control Theory}. Given a linear system $S$,
depending on a finite number of parameters. Determine
the control--theoretic properties (such as the decomposition
into a controllable and an autonomous part) of $S$ 
for all the values of involved parameters.\\

Since there is, for certain system classes \cite{CQR, LZ06, JFP},
a one--to--one correspondence
between control--theoretic properties of a system $S$
and the homological properties of an associated
module $M(S)$, we can reformulate the question as follows: \\

{\em Question for Computer Algebra}. Given a 
finite presentation of a parametric module $M$ over 
a (non--commutative) algebra $A$, determine the properties (e.g. homological)
of $M$ for all  the values of involved parameters. \\

We present detailed solutions for the bipendulum equations (Example \ref{Bipendulum}),
and for the "two pendula mounted on a cart'' problem, for both negligible friction (Section \ref{Nofriction})
and essential friction (Section \ref{Friction}). The latter problem, to the best of our knowledge, 
has not been yet solved completely in an explicit way. We also present and comment on
several curious examples. Namely, we show the existence of a non--generic controllability 
in the generically autonomous system (Example \ref{exGenAut}), and present a system, 
where both controllability and autonomy properties appear only in the non--generic situation (Example \ref{exNgenAC}).
In the treatment of the case 3 of \ref{Friction}, we illustrate the ability of our method to treat nested obstructions, that is, investigating sub--obstructions of a given obstruction to genericity. \\

{\bf Preliminaries:}
Algebraically speaking, a \textbf{parameter} is a non--zero 
 (and thus invertible) element of the ground field. 
In this article we deal with parameters which mutually 
commute with the elements of the algebra. In other words, the action of operators
of the algebra on parameters is just the commutative multiplication.

In this article, we use the following definition of a property being generic.

\begin{definition}
Let $\mathcal{P}(p_1,\ldots,p_n)$ be a polynomial expression in $p_i$ over some domain $\mathcal{D}$, on which 
a measure $\mu$ exists. The identity $\mathcal{P} = 0$
holds \textbf{generically} in $\mathcal{D}$ if $\mathcal{P}(\xi_1,\ldots,\xi_n)=0$, 
for almost all $(\xi_1,\ldots,\xi_n) \in \mathcal{D}^n$.
 
In other words, $\mathcal{P}=0$ holds in $\mathcal{D}^n \setminus E$,
where $E\subset \mathcal{D}^n$  and $\mu(E)=0$.
\end{definition}

For instance, if $\mathcal{D} = \C$ and $\mathcal{P}$ vanishes on the complement of
a nonzero algebraic set $E\subset \mathcal{D}^n$, then 
$\mathcal{P} = 0$ holds generically in $\mathcal{D}^n$. 
The Krull dimension $\krdim$ of the coordinate ring $\K[V]$ of a variety $V$ can be used for defining a measure on closed subsets $E \subset \mathcal{D}^n$ by assigning 
$\mu(E) = 0$ if $\krdim \K[E] < n$, and $\mu(E)=1$ otherwise.\\

{\bf Notations:} For a matrix $M$, $M^T$ denotes its transposed matrix. By 
${}_{A} \langle F \rangle$ we denote a left $A$--submodule, generated by the finite set $F$. The
subscript $A$ is dropped when $A$ is commutative. For an ideal $I$ in a commutative
ring $\K[x_1,\ldots,x_n]$, we denote by $V(I)\subseteq \K^n$ the set of common zeros of polynomials in $I$.

\section{Genericity of Gr\"obner Bases of Parametric Modules}

Since the major role in computations (of e.g.\ homological properties) 
is played by Gr\"obner bases, we investigate their behaviour in the
case when a ground field involves parameters.

Let $\K$ be a field. Let $A$ be a (non--commutative) algebra over $\K(p_1,\ldots,p_t)$. 
Suppose that in this algebra the notion of algorithmic left Gr\"obner basis
exists (e.g. $A$ can be a ring of solvable type \cite{Kr} or, more restrictively,
an Ore algebra \cite{CQR}).

Let us recall the definition of a ring of solvable type.

\begin{definition}
Let $K \supseteq \K$ be a skew field and let $R':=K[x_1,\ldots,x_n]$ be a commutative ring over $K$. 
Suppose that $\prec$ is a fixed term ordering on $R'$. 
Let $R$ be a ring generated by $\{x_1,\ldots,x_n\}$ subject to the new multiplication $*$.  
If the properties 1 and 2 below hold and $(R,*)$ is an associative ring, $R$ is called a \textbf{ring of solvable type}.


\textbf{1.} $\forall 1\leq i<j \leq n$, $x_i * x_j = x_i x_j$ and $x_j*x_i = c_{ij} x_i x_j + p_{ij}$, where $0\not=c_{ij}\in K$ and $p_{ij}\in R'$,
such that $\lm(p_{ij}) \prec x_i x_j$,

\textbf{2.} $\forall 1\leq i \leq n$, $\forall a\in K$, $a* x_i = a x_i$ and $x_i*a = c_{ai} a x_i + p_{ai}$, where $0\not=c_{ai}\in K$ and $p_{ai}\in K$.
\end{definition}

Good examples of rings of solvable type are the rings of (partial) differential--difference operators. \\

The elements of a free module $A^m$ are represented as the vectors 
$\bar{t} = \sum_{k=1}^m t_i e_i$, where $t_i \in A$, and $e_i$ is the $i$--th canonical basis vector. By $\bar{0}$ we denote the zero vector 
$(0,\ldots,0)^T \in A^m$. The set of vectors $\bar{t}_1,\ldots,\bar{t}_l$, for instance the set of generators of a submodule of a free module $A^m$, will be often identified with the matrix $T \subset \Mat(m \times l, A)$. A single vector $\bar{t}_i$ corresponds to the $i$--th column of $T$ and vice versa. \\

Given a monomial well--ordering $\prec$ on $A$, there are several ways to extend
it to a monomial module ordering $\prec_M$ on $A^m$, that is, an ordering consisting
of two components $(\prec, \prec_C)$, where $\prec_C$ is an ordering on
the components $e_i$. In the following, we need a so--called
\textbf{term over position ordering}, that is, $m_1 e_i \prec_M m_2 e_j $ if and only if $m_1 \prec m_2$ or, if $m_1 = m_2$, then $e_i \prec_C e_j$ for monomials $m_i \in A$.\\

Recall, that a \textbf{left syzygy} of a finite set of elements $\{f_1,\ldots,f_m\}$, $f_i \in A$, is a tuple $(b_1,\ldots,b_m)^T \in A^m$, such that $b_1 f_1 + \dots + b_m f_m = 0$. The set of all left syzygies of a given set of $m$ elements is a left submodule of $A^m$. It is often denoted as $\syz(\{f_1,\ldots,f_m\})$. \\

In this article, we work with left submodules, left syzygies etc. 
It is clearly possible to do the same also from the right. However, two--sided (bimodule) problems deserve,
except for the commutative case, a fairly distinct treatment. Most (if not all) problems, originating from 
applications of e.g.\ control theory, are formulated in terms of left modules.

\subsection{\textsc{Lift} and \textsc{LeftInverse} Algorithms}

\begin{proposition}
\label{PropLift}
Suppose that a left submodule $L$ of a free module $A^m$ is generated by
the set of column vectors $F=\{\bar{f}_1,\ldots,\bar{f}_l\}\subset A^m$. 
Consider the set $\widetilde{F} := \{\bar{f}_1 + e_{m+1}, \ldots, \bar{f}_l + e_{m+l}\}$
and assume, that the fixed ordering $\prec$ on $A^m$, naturally extended to the ordering $\prec_l$ on $A^{m+l}$,
satisfies $x^{\alpha} e_{m+i} \prec_l x^{\beta} e_j$, for all $1\leq i \leq l$, $1\leq j\leq m$ and for all $\alpha, \beta$.
Suppose that the 
left Gr\"obner basis $\widetilde{G}$ of $\widetilde{F}$ is finite. Then 
we reorder the columns of $\widetilde{G}$ in such a way, that the
elements, whose first $m$ components are zero, are moved to the left.
This process is schematically presented in the following picture:
\[
\widetilde{F} = 
\left(
\begin{array}{ccc}
\bar{f}_1 & \ldots & \bar{f}_l  \\
1   &        & 0  \\
    & \ddots & \\ 
0 & & 1 \\ 
\end{array}\right)
\; \overset{\textsc{leftGB}}{\longrightarrow} \;
\left(
\begin{array}{ccc|ccc}
\bar{0} & \ldots & \bar{0} \ & \bar{h}_1 & \ldots & \bar{h}_t \\
\hline
  & & \  &     &        & \\
  & {\displaystyle \bf S} & \  &     & {\displaystyle \bf  T}       & \\
  & & \  &     &        & \\
\end{array}\right)
= \widetilde{G}.
\]

Let $H=\{\bar{h}_1, \ldots, \bar{h}_t\}$ be a left Gr\"obner basis of $F$. 
Recall that we identify $F$ with the matrix $(\bar{f}_1,\ldots,\bar{f}_l)\in A^{m\times l}$,
 $T$ with an ${l \times t}$ matrix over $A$, and $H$
 with the matrix $(\bar{h}_1, \ldots, \bar{h}_t)\in A^{m\times t}$. Then

$\bullet$  $\mathbf{T}$ is a left transformation matrix between two generating sets of $F$, 

\hspace*{0.2cm} that is $H^T = \mathbf{T}^T F^T$ holds,

$\bullet$ the columns of $\mathbf{S}$ form a left Gr\"obner basis of \; $\syz(\{\bar{f}_1, \ldots, \bar{f}_l\})$.
\end{proposition}

\begin{proof}

\[
\text{Since } \bar{h}_i = \sum_{k=1}^l a_{ik} \bar{f}_k, \text{ we have }
\sum_{k=1}^l a_{ik} (\bar{f}_k+e_{m+k}) = \bar{h}_i + \sum_{k=1}^l a_{ik} e_{m+k}.
\]

Hence, the $i$--th column of $\bf{T}$ is $(a_{i1},\ldots,a_{il})^T$, and 
$H^T = (\bar{h}_1, \ldots,\bar{h}_t)^T = \mathbf{T}^T \cdot F^T$.\\

Let $\mathbf{S} = \{ \bar{s}_1, \ldots, \bar{s}_r\}$ and 
$\widetilde{S} := \{ (\bar{0},\bar{s}_1)^T, \ldots, (\bar{0},\bar{s}_r)^T\}$.
Since $\widetilde{G}$ is a left Gr\"obner basis of $\widetilde{F}$,
for any $f$ in ${}_{A} \langle \widetilde{F} \rangle \cap (\{\bar{0}\}^m \times A^l)$ 
$= {}_{A} \langle \widetilde{F} \rangle \cap \oplus_{k=m+1}^{m+l} A e_k$
there exists $g \in \widetilde{G}$, such that $\lm(g)$ divides $\lm(f)$. Then 
$\lm(g) \in \{\bar{0}\}^m \times A^l$, hence, by the property of the ordering, 
$g \in \widetilde{G} \cap (\{0\}\times A^l)  = \widetilde{S}$. Thus 
$\widetilde{S}$ is a left Gr\"obner basis of $\widetilde{F} \cap (\{\bar{0}\}^m \times A^l)$
and, in particular, $\widetilde{S}$ generates the latter. Since

\[
\sum_{k=1}^l b_{k} (\bar{f}_k+e_{m+k}) = \sum_{k=1}^l b_{k} e_{m+k} 
\; \text{ holds if and only if } \; \sum_{k=1}^l b_{k} \bar{f}_k = \bar{0},
\]

$\mathbf{S}$ consists of columns $(b_1,\ldots,b_l)^T$, which are the syzygies of the set $\{\bar{f}_1, \ldots, \bar{f}_l\}$.
\end{proof}

\begin{remark}
Clearly, for a Noetherian algebra $A$ the algorithm terminates.

More generally, if the left Gr\"obner basis of $F$ is finite, 
we get the transformation matrix in finitely many steps. Namely,
in the generalized Buchberger's algorithm for computing left Gr\"obner basis, 
we do not consider $S$--polynomials between elements whose leading monomials include components greater than $m$.

If the algebra $A$ is commutative, the transformation matrix property translates
into $H = F \cdot \mathbf{T}$.
\end{remark}


We call the algorithm computing the transformation matrix as above \textsc{Lift}$(F,H)$. 
Note that with this algorithm we are able to trace any computation which uses Gr\"obner bases. It is worth mentioning that Proposition \ref{PropLift} shows, that with basically one Gr\"obner basis computation we can get three important objects, namely a Gr\"obner basis of a module, a Gr\"obner basis of the first syzygy module and a transformation matrix. These three applications are sometimes called \textbf{Gr\"obner trinity} and play a fundamental role in computer algebra.

%

Many problems in control theory involve parameters, which are known to 
be non--zero, or even strictly positive, for physical reasons. However,
it might happen that the vanishing of certain algebraic 
expressions in the parameters has a direct impact on the
control--theoretic properties. Very often we observe
generically controllable parametric systems which, for some values of parameters, become uncontrollable.

As a further application of the algorithm \textsc{Lift}, we compute a left inverse of a given polynomial matrix in the case it exists. Below, the algorithm \textsc{rmLeftGroebnerBasis}$(M)$ computes the monic reduced minimal left Gr\"obner basis of a submodule $M$, which is unique for a fixed ordering (\cite{Kr, LVdiss}).

\newpage

\begin{algorithm}
\textsc{LeftInverse}(matrix M) \\
\textit{Input:} ${ }_{ }$ $M \in \Mat_{ m\times n}(A)$ \\
\textit{Output:} $L \in \Mat_{ n\times m}(A)$, such that $L \cdot M = Id_{n \times n}$

\hspace{1.1cm} or  $0\in \Mat_{1\times 1}(A)$, if no left inverse exists\\
${}_{}$\\
module $G := $ \textsc{rmLeftGroebnerBasis}$(M)$\\
{\bf if} $G \not = Id_{n \times n}$ {\bf then}

\hspace{0.6cm} {\bf report} "No left inverse exists" 

\hspace{0.7cm}{\bf return} $0$\\
{\bf endif}\\
${}_{}$\\
module $N := $ \textsc{Transpose}($M$) \\
module $K := $ \textsc{Lift}($N,Id_{n \times n})$ \\
{\bf return} \textsc{Transpose}($K$)
\end{algorithm}

\begin{proof}
The algorithm \textsc{LeftInverse} terminates as soon as \textsc{Lift} does.
Note that $LM = Id_{n \times n}$ can happen only in the case  
when the monic reduced minimal left Gr\"obner basis of a free submodule generated by the columns $\{N_j\}$ of $N=M^T$ is equal to $Id_{n \times n}$. 

In the setup of the \textsc{Lift} algorithm, we use $H = Id_{n \times n}$. Denote by $K$
the result of \textsc{Lift}($N,Id_{n \times n})$. Then, by the Proposition \ref{PropLift}, 
$Id_{n \times n} = Id_{n \times n}^T = K^T M^T = K^T N$. Hence, for $L =  K^T$ we have $LM=Id_{n \times n}$.
\end{proof}

The existence of a left inverse (or, more generally, a generalized inverse $G$, such that
$G \cdot M \cdot G = G$) often gives us the information on genericity of parameters. Namely,
one analyzes the possible vanishing of denominators of a generalized inverse, as it is done in e.g. \cite{CQR}. In the special case where 
$A = \K[\d]$ is a principal ideal domain, consider the module $M = A^{1 \times q} / A^{1 \times p} R$. Without loss of generality, 
we can assume $R$ has full row rank. Then $M$ is torsion--free if and only if there exists a right inverse to $R$.

As we have shown, computing the inverse is a special case of computing the transformation matrix with the algorithm \textsc{Lift}. In comparison with \textsc{LeftInverse}, \textsc{Lift} allows us to deal effectively with more general problems.

We call the polynomials in parameters, whose vanishing implies the failure of generic properties, \textbf{obstructions to genericity}. We can compute them as described above using the \textsc{Lift} algorithm. 



There is a need for complete information on the parametric module. It consists 
of the list of properties, computed for the generic and all the non--generic cases.
In the context of generically controllable problems, we are interested in computing e.g. an
annihilator of a torsion submodule for the each non--generic case. Thus, we need to stratify
the set of obstructions.

\subsection{Stratification of Obstructions to Genericity}

Let $\K$ be a field of characteristic 0.
Recall that a set is called \textbf{locally closed}, if it is a difference of two closed sets. A finite union of locally closed sets is called a \textbf{constructible} set.

Suppose we are given a set of polynomials $P = \{p_1,\ldots,p_n\} \subset \K[a_1,\ldots,a_m]$,
which are irreducible over $\K$.

We associate to $P$ a set $C(P):=\{ \bar{\xi}= (\xi_1,\ldots,\xi_m) \in \K^m \mid \prod_{i=1}^n p_i (\bar{\xi}) = 0\}$.

\newpage

\begin{lemma}
\label{Strat}
The set $C(P)$ is constructible. 
\end{lemma}

\begin{proof}

Let $\Omega :=  \{ (\Lambda', \Lambda'') \mid \Lambda' \cup  \Lambda'' = \{1,\ldots,n\}, \Lambda' \cap  \Lambda''=\emptyset\}$ be the set of all divisions of
$\{1,\ldots,n\}$ into two disjoint complementary subsets.
Let, furthermore, $\Sigma := \Omega \setminus (\emptyset,\{1,\ldots,n\})$. 
Then, 
\[
C(P) = \bigcup_{(j,k)\in \Sigma} \{ \bar{\xi} \mid \forall j \in \Lambda' \; p_j( \bar{\xi})=0, \quad \forall k \in \Lambda'', \; p_k( \bar{\xi}) \not=0 \} =
\]
\[
= \bigcup_{(j,k)\in \Sigma} 
V(\langle \{ p_j \mid j \in \Lambda' \} \rangle) \setminus 
V(\langle \{ p_k \mid k \in \Lambda'' \} \rangle) = 
\]
\[
= \bigcup_{(j,k)\in \Sigma} 
 \cap_{j \in \Lambda'} V(\langle p_j \rangle) \setminus 
\cap_{k \in \Lambda''} V(\langle p_k \rangle),
\]

and, indeed, we see that $C(P)$ is a disjoint union of locally closed sets.
Note that in $C(P)$ there is a closed subset $\cap_i V(p_i)$; the rest of subsets are locally closed.
\end{proof}

It is convenient to represent $C(P)$ as a binary tree, where the vertices 
are the decision points, associated to polynomials $p_i$, and the
edges represent the logical conditions $(p_i = 0)$ and $(p_i \not = 0)$, respectively.
In such a way it is easy to see, that starting from $n$ elements in
the set $P$, we will have $2^n - 1$ algebraic systems describing
the locally closed components of $C(P)$. 

Given two ideals $I,J \in \K[a_1,\ldots,a_m]$, an algebraic data describing a locally closed set $V(I) \setminus V(J) = V(I) \setminus (V(I) \cap V(J))$ can be computed with a factorizing Gr\"obner basis algorithm (e.g. \cite{GPS}). Such an algorithm takes $I,J$ as input and returns a list of ideals, where the zero set of the intersection of the output ideals is contained in the $V(I)$ and contains the complement of the $V(J)$ in $V(I)$. We refer to this algorithm as to \textsc{FactGB}$(I,J)$.

\begin{example}
Let $P = \{p_1,p_2 \}$, then the binary tree for $C(P)$ consists of the following 3
systems of equations and inequations:
$\{ p_1 =0, p_2 = 0 \}$, $\{ p_1 \not=0, p_2 = 0 \}$ and
$\{ p_1 =0, p_2 \not= 0 \}$.

Denote $V_i:=V(\langle p_i \rangle)$ for $i=1,2$, and $V_{12} := V(\langle p_1,p_2 \rangle) = V_1 \cap V_2$. Then, the decomposition of $C(P)$ can be written as 
$V_{12} \uplus (V_1 \setminus V_{2})  \uplus (V_2 \setminus V_{1}) = V_{12} \uplus (V_1 \setminus V_{12})  \uplus (V_2 \setminus V_{12})$, where $\uplus$ denotes the disjoint union.

Computationally, we need to compute the Gr\"obner basis of an ideal $I_{12}:=\langle p_1,p_2 \rangle$, and two lists $L_i :=$ \textsc{FactGB}$( \{ p_i \}, I_{12})$, obtained with the factorizing Gr\"obner basis algorithm, which describe $V_i \setminus V_{12}$.

\end{example}

Given a set of polynomials $\{f_1,\ldots,f_s\}\subset \K[a_1,\ldots,a_m]$, we factorize them and form a set of pairwise different irreducible factors $P:=\{p_1,\ldots,p_n\}$.
We sort $p_i$ by using a positively graded degree ordering, starting with the smaller elements. With such an ordering, it is easier to compute with locally closed sets.
Namely, the bigger elements will often reduce to simpler polynomials with respect to the smaller elements. Thus, also the detection of empty components 
(that is, systems with no solutions) can be achieved faster.

Lemma \ref{Strat} is constructive indeed. Together with the presentation of locally closed sets using the algorithm \textsc{FactGB} above, we call the whole procedure 
\textsc{StratifyLC}(list $L$). It takes a finite list of irreducible polynomials on the input
and returns a list of systems of equations and inequations, corresponding to $C(P)$.




\subsection{The \textsc{Genericity} Algorithm}

Let $A$ be a $\K$--algebra and suppose that the coefficients of a given system $S$ involve parameters $p_1,\ldots,p_t$.
We interpret the parameters as generators of the transcendental field extension of $\K$ and 
we use the natural $\K(p_1,\ldots,p_t)$--algebra structure on $A$.

\begin{algorithm}
\label{AlgGenericity}
\textsc{Genericity}(matrix M) \\
Assume, that a monomial module ordering on the algebra $A$ is fixed. \\
\textit{Input:} ${ }_{ }$ $M \in \Mat_{ m\times n}(A)$ \\
\textit{Output:} $\{h_1,\ldots, h_s\} \subset \K[p_1,\ldots,p_t]$, such that if a specialization of the parameters implies $h_i(p_1,\ldots,p_t) = 0$, then a left Gr\"obner basis of $M$ is different from the generic one \\
${}_{}$\\
module $G := $ \textsc{rmLeftGroebnerBasis}$(M)$ \qquad \verb?\\? $G=\{\bar{g}_1,\ldots,\bar{g}_{\ell}\} \in \Mat_{ m\times \ell}(A)$ \\
matrix $T := $ \textsc{Lift}($M,G)$ \qquad \qquad \qquad \qquad \qquad \quad \verb?\\? $T \in \Mat_{ n\times \ell}(A)$ \\
list $S$, $H$; \; int $i,j$ \\ 
for $j=1$ to $\ell$
\begin{itemize}
\item[] $i:=$ the leading component of $\bar{g}_j$
\item[] for $k=1$ to $n$ 
\begin{itemize}
\item[] if $(M_{ik} \not=0)$ then
\item[] \qquad if $(T_{kj} \not=0)$ then \; $S:= S \cup $ \textsc{Denominator}$(T_{kj})$
\item[] \qquad end if
\item[] end if
\end{itemize}
\item[] end for
\end{itemize}
end for \\
if ($S\not=$ empty list) then
\begin{itemize}
\item[] $H := $ \textsc{Factorize}$(S)$
\item[] $H := $ \textsc{Simplify}$(H)$
\end{itemize}
else $H := $ empty list \\
{\bf return} $H$
\end{algorithm}

\begin{proof}
The algorithm \textsc{Factorize}(list $L$) returns a list of monic factors of
every polynomial of the list $L$. The algorithm \textsc{Simplify}(list $L$)
 refines a list $L$ by removing doubled appearances of same elements. We may assume it also
 sorts $L$ by an ordering, putting with the smaller elements in the beginning of the output.

The algorithm \textsc{Genericity} terminates as soon as \textsc{Lift} terminates.
Now, we prove the correctness. Suppose that the leading term of $\bar{g}_j$ lies in the $i$--th
module component. From the property $G^T = T^T M^T$ it follows, that there is a 
presentation of the element $G_{ij} \in A$ as the sum 
\[
G_{ij} = \sum_{k=1}^{n} T_{kj} M_{ik} = \sum_{k=1}^{n} T_{jk}^T M_{ki}^T.
\]

Hence, it suffices to collect only the 
denominators of $T_{kj} \not=0$ with $M_{ik} \not=0$,
since only such elements contribute to the leading coefficient of $\bar{g}_{j}$.

If some leading coefficient of the unique generic Gr\"obner basis vanishes for some specialization of parameters, then the Gr\"obner basis under such a specialization is different from the generic one. 
\end{proof}

Note that with the algorithm we obtain the expressions in the parameters which lead to non--generic Gr\"obner bases. In order to obtain Gr\"obner bases under specialization, provided by $h_i$, one cannot use the generic Gr\"obner basis. Instead, one has to compute the specialized Gr\"obner basis from scratch.

Suppose that the output of \textsc{Genericity} is the list of irreducible polynomials $H$. In practice, we exclude from $H$ the polynomials, which do not satisfy the problem--specific constraints for e.g. physical admissibility like non--negativity. 
Then, we apply the algorithm \textsc{StratifyLC}($H$) and 
obtain a complete stratification of a given system with respect to its parameters. 


\subsection{Comparison with Other Methods}

\subsubsection{Comprehensive Gr\"obner bases}

Comprehensive Gr\"obner bases (see e.g. \cite{Weispfenning2003}) were introduced by Weispfenning and generalized to rings of solvable type by Kredel \cite{Kr}.

A comprehensive Gr\"obner basis, by definition, is a finite subset $G$ of a parametric polynomial ideal $I$ such that  $\sigma(G)$ constitutes a Gr\"obner basis of the ideal generated by $\sigma(I)$ under all specializations $\sigma$ of the parameters in arbitrary fields (\cite{Weispfenning2003}).\\

The construction of a comprehensive Gr\"obner basis follows the lines of Buchberger's algorithm. However, the result will be a union of trees of ideal bases (called Gr\"obner systems), where each basis is accompanied with a set of conditions of parameters. Being a powerful theoretical instrument, comprehensive Gr\"obner bases are quite complicated to compute. To the best of our knowledge, there is no implementation yet, which is able to treat serious examples.\\

In our approach we separate two processes, which are unified in the comprehensive Gr\"obner basis method. Namely, we compute the tree of sets of conditions of parameters \textit{after} the Gr\"obner basis and transformation matrix computations. In such a way we avoid repeated computations in trees of ideals and sets of conditions, which might occur during the computation of a comprehensive Gr\"obner basis. 

\subsubsection{The Leykin--Walther Method}

The method has been formalized by Leykin for the case of ideals \cite{Leykin01} and has
been generalized to modules by Walther \cite{W03}. The idea behind the method has
been used before, however Leykin and Walther formulated and proved the whole framework in
a complete way. In the following, we reformulate the Lemma 2.3 from \cite{W03}.\\

Let $\K$ be a field of characteristic 0.
Given a $\K$--algebra $A$, we consider parameters as new commutative variables and perform
further computations in the $\K$--algebra $\widetilde{A}:=A \otimes_{\K} \K[p_1,\ldots,p_m]$.
We use in $\widetilde{A}$ an elimination ordering $\prec_A$ for the variables of $A$. 
Such an ordering is characterized by the property $p_1^{\alpha_1} \dots p_m^{\alpha_m} \prec_A t$, for any monomial $t\in A$ and any $\alpha \in \mathbb{N}^m$. \\

Let $G = \{g_1,\ldots,g_{\ell}\}$ be a reduced Gr\"obner basis for the left submodule 
$N \subset \widetilde{A}^s$ with respect the position over term ordering, induced by $\prec_A$ on $\widetilde{A}^s$. Let, moreover, $Q_N \subset \K[p_1,\ldots,p_t]$ be the ideal 
$\{ p \in \K[p_1,\ldots,p_t] \mid p\widetilde{A}^s \subseteq N\}$. 
For $g_i \not= Q_N \widetilde{A}^s$, multiply all the leading coefficients with respect to $\K[p_1,\ldots,p_t]$ of such $g_i$ and denote the result by $h$. Let $\sigma: \K[p_1,\ldots,p_t] \to \K$ be a specialization, then if $\sigma(h)\not=0$, then $\sigma(G) = \{\sigma(g_1),\ldots,\sigma(g_t)\}$ is a Gr\"obner basis. \\

This method has some drawbacks in practice. Suppose that the number of parameters is big and there are many obstructions, which appear in several components as, say, leading coefficients by a monomial $1$. This situation is typical for generically controllable systems. Then, using the method of Leykin--Walther, we are forced to compute Gr\"obner basis of a submodule of elements as described above, whereas a better solution would be just to collect the leading coefficients in parameters. Secondly, in a similar situation we get many elements in Gr\"obner basis and the analysis of the impact of obstructions, e.g. the computation of the stratification, becomes very involved. \\

On the other hand, this method allows us to handle the cases, when the parameters
satisfy algebraic identities between themselves or when there are more general identities,
involving both variables and parameters. We believe, that this method will be enhanced in
order to overcome the described difficulties.



\section{Implementation of Algorithms}

The described method for detecting the obstructions to
genericity of parametric modules is implemented in the procedure \texttt{genericity} of control theory toolbox \textsc{control.lib} \cite{Controllib}, which is realized as a library
in the computer algebra system \textsc{Singular} \cite{Singular}. \textsc{Singular} is the specialized computer algebra system for polynomial computations, well--known for its high performance (especially in Gr\"obner bases--related computations) and rich functionality. It uses intuitive C--like programming language, in which the libraries are written. It is important to mention, that \textsc{Singular} is distributed under GPL license, that is, it is free for academic purposes. \\

The current implementation of the procedure \texttt{genericity} works in a little different way, compared with the Algorithm \ref{AlgGenericity}. Namely, it takes as input a matrix $T$, which is assumed to be the result of the \textsc{Lift} algorithm.
This minor modification allows us to compute the data, which are independent from the choice of a monomial module ordering.
The output of the procedure \texttt{genericity} is a list of strings and thus it is ring--independent.
In the first item of the list the names of parameters, by which we have divided in the algorithm, are collected. Every further item of the list contains a single non--trivial polynomial in the parameters.


There are several algorithms in \textsc{Singular}, which compute (left) Gr\"obner bases 
of modules over commutative polynomial algebras and non--commutative $GR$--algebras \cite{LVdiss, LS}.
It is recommended to use the heuristic routine \texttt{groebner}, which often provides the best match for a concrete example. For more details on \textsc{Singular}, consult with the book \cite{GPS} and with the website of the system \cite{Singular}, which contains among other the online documentation. The algorithm \textsc{FactGB} is implemented in \textsc{Singular} and is accessible via the function \texttt{facstd}. 


In the library \textsc{control.lib}, we have implemented several functions for 
supporting the research in systems and control theory. Among
others, there are the procedures \texttt{LeftInverse} and \texttt{LeftKernel}, 
their counterparts
\texttt{RightInverse} and  \texttt{RightKernel}, as well as \texttt{canonize} 
and \texttt{iostruct}.

The main purpose of the library is to provide maximal relevant information 
based on a simple input. This principle led us to the development of heuristic procedures
\texttt{control} and \texttt{autonom}, which use homological computations. 
Respectively, for systems with a full row rank presentation matrix, 
there are dimension--guided procedures \texttt{controlDim} and \texttt{autonomDim}. 

Given a system algebra and a system module over it, both procedures 
compute relevant properties of a given module from the point of view of controllability
(with the procedure \texttt{control} or \texttt{controlDim}) or autonomy analysis (with the 
procedure \texttt{autonom} or \texttt{autonomDim}). The procedure \texttt{canonize} takes
the output of either \texttt{control} or \texttt{autonom} procedure and computes reduced and tail--reduced Gr\"obner bases of the objects,
thus simplifying and \textit{canonizing} the output.

We illustrate the functionality of the library and the flexibility of \textsc{Singular} with the following example.

\begin{example}
\label{Bipendulum}
Consider a bipendulum, that is, a system, describing a bar with two fixed pendula of length $\ell_1$ and $\ell_2$ respectively (e.g. \cite{JFP, OreModules}). The system algebra is a commutative algebra in variable $\d$ over a field of rationals with parameters $g, \ell_1, \ell_2$, that is, $\Q(g, \ell_1, \ell_2)[\d]$.
A system module is presented via the matrix
$\begin{pmatrix} \d^2+\frac{g}{\ell_1} & 0 & -\frac{g}{\ell_1} \\ 0 & \d^2+\frac{g}{\ell_2} & -\frac{g}{\ell_2} \end{pmatrix}$. We run the following code in a \textsc{Singular} session.
\small
\begin{verbatim}
LIB "control.lib";
option(redSB); option(redTail);
\end{verbatim}
\normalsize

With the \texttt{LIB} command we load the library. The \texttt{option} commands
tell \textsc{Singular} to compute reduced bases (\texttt{option(redSB)}),
and also reduce not only leading terms, but any terms in the occurring polynomials
 (\texttt{option(redTail)}). \\

It is important to mention, that any polynomial computation in \textsc{Singular} requires the
definition of a ground ring.
\small
\begin{verbatim}
ring r1 = (0,g,l1,l2),(d),(c,dp);
module RR = [d^2+g/l1, 0, -g/l1], [0, d^2+g/l2, -g/l2];
\end{verbatim}
\normalsize

The ring we set bears the name \texttt{r1}, it has $\Q(g, \ell_1, \ell_2)$ as the ground field 
(0 stands for the characteristic of a field, \texttt{g, l1, l2} is a list of names for parameters), and the only variable \texttt{d}. The last
comma--separated block describes the monomial module ordering on \texttt{r1}. In this case
\texttt{(c,dp)} means the following. We use the descending ordering \texttt{c} on the module components 
 and the degree reverse lexicographical ordering \texttt{dp} on the monomials in the same component.
\small
\begin{verbatim}
module R = transpose(RR);
list L = canonize(control(R));
L;
\end{verbatim}
\normalsize

We have to transpose the module $R$, because \textsc{Singular} takes the columns of 
a given matrix presentation as the generators of a module. Here is the output
of \textsc{Singular}:
\small
\begin{verbatim}
[1]:
   number of first nonzero Ext:
[2]:
   -1
[3]:
   strongly controllable(flat), image representation:
[4]:
   _[1]=[(-g*l2)*d^2+(-g^2),(-g*l1)*d^2+(-g^2),
         (-l1*l2)*d^4+(-g*l1-g*l2)*d^2+(-g^2)]
[5]:
   left inverse to image representation:
[6]:
   _[1,1]=(-l1)/(g^2*l1-g^2*l2)
   _[1,2]=(l2)/(g^2*l1-g^2*l2)
   _[1,3]=0
[7]:
   dimension of the system:
[8]:
   1
[9]:
   Parameter constellations which might lead to a non-controllable system:
[10]:
   [1]:
      g
   [2]:
      l1-l2
\end{verbatim}
\normalsize

As one can see, in the output of the procedure we provide both textual comments on the properties of a system
and the corresponding data. The heuristics says that the modules $\Ext^i_{r1}(R,r1)$
of a transposed module indeed vanish for $i\geq 1$ ($-1$ is returned in this situation). Hence, the system is generically controllable
(the notion of strong controllability from above coincides with classical controllability for systems of ordinary differential equations).
Moreover, the procedure computes the image representation, left inverse to the image representation and the dimension of the system. The 10-th item is the output of the procedure \texttt{genericity}, that is, a list of strings. 
The polynomial obstruction to genericity in this example is $\ell_1 - \ell_2$. The monomial
obstruction $g$ is not physically admissible. \\

Let us analyze the properties of the system in the non--generic case $\ell_1 = \ell_2 = \ell$.
We do this with the help of the following code in the same \textsc{Singular} session:
\small
\begin{verbatim}
ring r2 = (0,g,l),(d),(c,dp);
module RR = [d^2+g/l, 0, -g/l], [0, d^2+g/l, -g/l];
module R = transpose(RR);
list L = canonize(control(R));
L;
\end{verbatim}
\normalsize

We get the following output:
\small
\begin{verbatim}
[1]:
   number of first nonzero Ext:
[2]:
   1
[3]:
   not controllable , image representation for controllable part:
[4]:
   _[1]=[(g),(g),(l)*d2+(g)]
[5]:
   kernel representation for controllable part:
[6]:
   _[1]=[0,1]
   _[2]=[1]
[7]:
   obstruction to controllability
[8]:
   _[1]=[0,1]
   _[2]=[(-l)*d2+(-g)]
[9]:
   annihilator of torsion module (of obstruction to controllability)
[10]:
   _[1]=(-l)*d2+(-g)
[11]:
   dimension of the system:
[12]:
   1
\end{verbatim}
\normalsize

We see that the system is not controllable, since it contains a
torsion submodule annihilated by $\langle \ell \d^2 + g \rangle$. However,
we give both image and kernel representations for the controllable part
of the system and describe the  obstruction to controllability explicitly.
Now, we are interested in the autonomy analysis of this non--controllable
system, what can be achieved with the following code:
\small
\begin{verbatim}
list A = canonize(autonom(R));
A;
\end{verbatim}
\normalsize

This gives us the following output:
\small
\begin{verbatim}
[1]:
   number of first nonzero Ext:
[2]:
   0
[3]:
   not autonomous
[4]:
   kernel representation for controllable part
[5]:
   _[1]=[0,1]
   _[2]=[(-l)*d2+(-g),-1]
   _[3]=[(g)]
[6]:
   column rank of the matrix
[7]:
   2
[8]:
   dimension of the system:
[9]:
   1
\end{verbatim}
\normalsize

Since the $0$--th $\Ext$ module of the system module $RR$ (in other words, $\Hom_{r2}(RR,r2)$) does not vanish, the
system is not autonomous. In addition, we compute a kernel representation for the controllable part,
the column rank of the presentation matrix and the dimension of the system.
\end{example}

Parametric systems quite often are generically controllable and
contain an autonomous subsystem for some special values of parameters.
In the following example, we show that also a generically autonomous
system might be controllable in a non--generic case.

\newpage

\begin{example}
\label{exGenAut}
Let $R=\K(a,b)[\d]$ be a ring. A module $N = R/ \langle a \d + b \rangle$ is generically
autonomous. However, 

if $a=0, b\not=0$, then $M=0$ and thus $M$ is autonomous, \\

if $a=0, b=0$, then $M$ is free of rank 1 and hence $M$ is controllable.

\end{example}

A general system might specialize to controllable and
autonomous system in non--generic cases, as the next example shows.

\begin{example}
\label{exNgenAC}

Let $R=\K(a,b)[\d]$ be a ring. Consider a module 
$M = R^2/ \begin{pmatrix} 0 & 0 \\ 0 & a \d + b \end{pmatrix}$.
Generically, it is neither controllable nor autonomous, the
annihilator of a torsion submodule is $\langle a \d + b \rangle$.

The stratification of $M$ with respect to parameters looks as follows:

if $a\not=0, b=0$, a torsion submodule of $M$ is annihilated by $\langle \d \rangle$, \\

if $a=0, b\not=0$, then $M$ is free of rank 1, \\

if $a=0, b=0$, then $M$ is free of rank 2. \\

Assume, that $a,b \in \mathcal{D} \supseteq \K$. Then the space of parameters $\mathcal{D}^2$ decomposes into a
direct sum of subspaces $G \uplus E_1 \uplus E_2 \uplus E_3$, where
$G = \{(a,b) \mid a\not =0, b\not=0\}$,
$E_1 = \{(a,b) \mid a=0, b\not =0\}$,
$E_2 = \{(a,b) \mid a\not =0, b=0\}$ and
$E_3 = \{(a,b) \mid a =0, b =0\}$. Denote by $\bar{E}$ the closure of $E$, then $\dim \bar{E}_1 = \dim \bar{E}_2 = 1$, $\dim \bar{E}_3 = 0$ in $\mathcal{D}^2$. Hence, all $\bar{E}_i$ have measure 0 and $\bar{G} = \mathcal{D}^2$ has measure 1.
\end{example}

\begin{remark}
There are packages like $D$--modules for \textsc{Macaulay2}, \cite{dmodMac2}, and 
\textsc{OreModules} for \textsc{Maple}, \cite{OreModules}, which have a functionality
to treat some of the problems above. The latter package provides the possibility to reveal dangerous parametric denominators via the computation of generalized inverse.
\end{remark}

\section{Example: Two Pendula, Mounted on a Cart}

Consider the Example 5.2.28 from \cite{PW} (see also the examples and solutions to them in \cite{OreModules}) describing two pendula, mounted on a cart.
  

In this example, $m_i$ is the mass and $L_i$ is the length of the $i$--th pendula. 
Respectively, $k_i$ and $d_i$ are the coefficients, characterizing the friction 
at the joints of pendula. $M_0$ denotes the mass of the cart and $g$ is a gravitational constant. 
All these parameters can take only non--negative values.

Let us denote $z_i := k_i - m_i L_i g$ for $i=1,2$. Then the presentation matrix for a system module is constituted by the raws of the following matrix
$$
\begin{pmatrix}
m_{1}L_{1}\d^{2} & m_{2}L_{2}\d^{2} & (m_{1}+m_{2}+M_0)\d^{2} & -1 \\
m_{1}L_{1}^{2}\d^{2}+d_{1}\d+z_{1} & 0 & m_{1}L_{1}\d^{2} & 0 \\
0 & m_{2}L_{2}^{2}\d^{2}+d_{2}\d+z_{2} & m_{2}L_{2}\d^{2} & 0
\end{pmatrix}
$$

We take the transposed module of the matrix. 
It is convenient to consider the columns of the matrix above
as the generators of submodule of a free module. Since the 
last generator then is just $(-1,0,0)^T$, we perform
reduction and simplification of first components with respect to 
this generator. In such a way we obtain much easier presentation matrix.

\subsection{Negligible Friction}
\label{Nofriction}

Let us assume, that the friction is negligible (that is, $d_i=0$ and $k_i=0$). 
We get the simplified presentation matrix of the transposed module as follows:

$$\left(
\begin{array}{*{4}{c}}
L_1 \d^2 - g & 0 & m_{1}L_{1} \d^{2} \\
0 & L_2 \d^2 -g & m_{2}L_{2} \d^{2} 
\end{array}
\right)
$$

The generic reduced minimal Gr\"obner basis is the $2 \times 2$ identity matrix. 
With the \textsc{Lift} algorithm we obtain the transformation matrix

$$\left(
\begin{array}{*{4}{c}}
\dfrac{L_1 L_2}{g^2 L_1-g^2 L_2}\d^2-\dfrac{1}{g} & -\dfrac{m_1 L_1 L_2}{g^2 m_2 L_1-g^2 m_2 L_2} \d^2 \\
{}_{} & {}_{} \\
\dfrac{m_2 L_1 L_2}{g^2 m_1 L_1-g^2 m_1 L_2}\d^2 & -\dfrac{L_1 L_2}{g^2 L_1-g^2 L_2}\d^2-\dfrac{1}{g} \\
{}_{} & {}_{} \\
-\dfrac{L_1 L_2}{g^2 m_1 (L_1- L_2)} \d^2+\dfrac{L_1}{g m_1 (L_1 - L_2)} & \dfrac{L_1 L_2}{g^2 m_2 (L_1- L_2)} \d^2-\dfrac{L_2}{g m_2 (L_1- L_2)}
\end{array}
\right)
$$

Collecting the denominators, we can see that their $lcm$ is $m_1 m_2 g^2 (L_1 - L_2)$.
Since $m_i$ and $g$ are strictly positive, the only obstruction to genericity appears when $L_1 - L_2=0$.\\

Indeed, in the case $L_1 = L_2 = L$ the generic Gr\"obner basis is
$\begin{pmatrix}
0 &             1 \\
L \d^2-g & \frac{m_2}{m_1}
\end{pmatrix}$, hence the system is not controllable. The torsion submodule
is annihilated by the ideal $\langle L \d^2-g \rangle$, but the system is not
completely autonomous.

\subsection{Essential Friction}
\label{Friction}

Now, all the parameters are strictly positive. The simplified presentation matrix of the transposed module is the following

$$\left(
\begin{array}{*{4}{c}}
L_1 \d^2 + d'_{1}\d+z'_{1} & 0 & m_{1}L_{1} \d^{2} \\
0 & L_2 \d^2 + d'_{2}\d+z'_{2} & m_{2}L_{2} \d^{2} 
\end{array}
\right)
$$

where $z'_i := \frac{z_i}{m_i L_i} = \frac{k_i}{m_i L_i} - g$ and $d'_i := \frac{d_i}{m_i L_i}$
for $i=1,2$. \\

The generic reduced minimal Gr\"obner basis is the $2 \times 2$ identity matrix. The output of the Algorithm \ref{AlgGenericity} delivers the list of three polynomials $\{z'_1, z'_2, P\}$,
where 
$$P = L_{1}^{2}{z'}_{2}^{2}-2L_{1}L_{2}z'_{1}z'_{2}-L_{1}d'_{1}d'_{2}z'_{2}+L_{1}{d'}_{2}^{2}z'_{1}+L_{2}^{2}{z'}_{1}^{2}+L_{2}{d'}_{1}^{2}z'_{2}-L_{2}d'_{1}d'_{2}z'_{1}.$$

$z'_i = 0$ means, that $k_i = m_i L_i g$. This is physically admissible situation. Let us analyze $P$ for the admissibility. Indeed, $P$ is irreducible but it has a special form, namely
\begin{equation}
\label{Peq}
P = (L_2 z'_1-L_1 z'_2)^2 + (L_2 d'_1- L_1 d'_2)\cdot(d'_1 z'_2-d'_2 z'_1).
\end{equation}

In particular, $P$ vanishes if both $z'_1$ and $z'_2$ do, so $P$ is admissible.


The stratification consists of 6 cases, namely
\begin{enumerate}
\item $k_1 = m_1 L_1 g, k_2 = m_2 L_2 g, P = 0$
\item $k_1 = m_1 L_1 g, k_2 \not= m_2 L_2 g, P = 0$
\item $k_1 = m_1 L_1 g, k_2 \not= m_2 L_2 g, P \not= 0$
\item $k_1 \not= m_1 L_1 g, k_2 = m_2 L_2 g, P = 0$
\item $k_1 \not= m_1 L_1 g, k_2 = m_2 L_2 g, P \not= 0$
\item $k_1 \not= m_1 L_1 g, k_2 \not= m_2 L_2 g, P = 0$
\end{enumerate}

The setup for \textsc{Singular} treatment of the cases is
the following:

\small
\begin{verbatim}
LIB "control.lib";
ring T = (0,g),(m1,m2,L1,L2,d1,d2,k1,k2),dp;
poly P = k1^2*L2^4*m2^2-2*k1*k2*L1^2*L2^2*m1*m2-k1*d1*d2*L2^2*m2+
k1*d2^2*L1^2*m1+2*k1*g*L1^2*L2^3*m1*m2^2-2*k1*g*L1*L2^4*m1*m2^2+
k2^2*L1^4*m1^2+k2*d1^2*L2^2*m2-k2*d1*d2*L1^2*m1-2*k2*g*L1^4*L2*m1^2*m2+
2*k2*g*L1^3*L2^2*m1^2*m2-d1^2*g*L2^3*m2^2+d1*d2*g*L1^2*L2*m1*m2+
d1*d2*g*L1*L2^2*m1*m2-d2^2*g*L1^3*m1^2+g^2*L1^4*L2^2*m1^2*m2^2-
2*g^2*L1^3*L2^3*m1^2*m2^2+g^2*L1^2*L2^4*m1^2*m2^2;
poly z1 = k1 - m1*L1*g;
poly z2 = k2 - m2*L2*g;
\end{verbatim}
\normalsize

In particular, we can see the expression for $P$ in terms of original
variables. The name of a ring, where the interesting parameters live
as polynomials, is $T$. In \textsc{Singular}, we can switch between
different rings and also map objects. \\

Case 1). $k_1 = m_1 L_1 g, k_2 = m_2 L_2 g, P = 0$. \\
Note that these three equations describe an algebraic variety, that
is a closed set. The Gr\"obner basis of the ideal $k_1 - m_1 L_1 g, k_2 - m_2 L_2 g, P$
is $k_1 - m_1 L_1 g, k_2 - m_2 L_2 g$, since $P$ vanishes, when both
$k_1 = m_1 L_1 g$ and $k_2 = m_2 L_2 g$. Hence, it suffices to plug the
values for $k_i$ in the corresponding system. For this, we run the following code:
\small
\begin{verbatim}
ring r1 = (0,g,m1,m2,L1,L2,d1,d2,k1,k2),(d),(c,dp);
poly z1 = 0; poly z2 = 0;
module RR =  
          [m1*L1^2*d^2+d1*d+z1, 0, m1*L1*d^2],
          [0, m2*L2^2*d^2+d2*d+z2, m2*L2*d^2];
module R = transpose(RR);
list LC = canonize(control(R));
list LA = canonize(autonom(R));
\end{verbatim}
\normalsize

From the output of \texttt{control} and \texttt{autonom} procedures,
we conclude, that this system is neither controllable nor autonomous. 
In particular, the torsion submodule is annihilated by $\langle \d \rangle$. \\

Case 2). $k_1 = m_1 L_1 g, k_2 \not= m_2 L_2 g, P = 0$. \\
Here we deal with the locally closed set 
$V(\langle k_1 - m_1 L_1 g, P \rangle) \setminus V(\langle k_2 - m_2 L_2 g\rangle) $.
Using the following code, we get its better description. We employ a technical trick
by modifying a ground ring in such a way, that $k_i$ have priority over the rest
of polynomials. In such a way during the computations the relation $k_1 = m_1 L_1 g$
will be used as replacing $k_1$ with $m_1 L_1 g$. This is achieved by using
a different ordering like e.g. the \textbf{elimination} ordering (see e.g. \cite{GPS}) for $k_1,k_2$.
\small
\begin{verbatim}
ring T2 = (0,g),(k1,k2,m1,m2,L1,L2,d1,d2),(a(1,1),dp);
poly z1 = ...; poly z2 = ...; poly P = ...; // we copy them from above
ideal I2 = P,z1;   
I2 = groebner(I2); 
facstd(I2,z2);
\end{verbatim}
\normalsize

The output of \texttt{facstd} command gives us the only component
\small
\begin{verbatim}
[1]:
   _[1]=k1+(-g)*m1*L1
   _[2]=k2*m1^2*L1^4+(-g)*m1^2*m2*L1^4*L2+m2*L2^2*d1^2-m1*L1^2*d1*d2
\end{verbatim}
\normalsize

We are able to extract e.g. $k_2$ from the last equation explicitly:

\[
k_{2} = m_{2}L_{2}g+ \dfrac{m_{1}L_{1}^{2}d_{2}-m_{2}L_{2}^{2}d_{1}}{m_{1}^{2}L_{1}^{4}} d_1
\]

Alternatively, we can express $d_2$ in terms of variables $m_i, L_i, k_2, d_1$.

Computing with substitutions, we see that this system is neither 
controllable nor autonomous. 
The torsion submodule is annihilated by $\langle m_1 L_1^2 \d ^2 + d_1 \d \rangle$. \\

Case 3). $k_1 = m_1 L_1 g, k_2 \not= m_2 L_2 g, P \not= 0$.\\

We use the computations of the case 2 and describe a locally closed set via
the following system of equations and inequations
\[
k_1 = m_1 L_1 g, k_2 - m_2 L_2 g \not=0, k_2 - m_2 L_2 g \not= \frac{m_{1}L_{1}^{2}d_{2}-m_{2}L_{2}^{2}d_{1}}{m_{1}^{2}L_{1}^{4}} d_1
\]

In order to treat both inequations involving $k_2 - m_2 L_2 g$, we introduce
a new parameter $u$ (thus, $u$ is mutually non--zero in the ground field) and plug in the transposed system module the fake equation $k_2 - m_2 L_2 g = u$.

Also this system is generically neither controllable nor autonomous. 
The torsion submodule is annihilated by $\langle \d \rangle$. Compare
with the annihilator for the case 2, which is $\langle m_1 L_1^2 \d ^2 + d_1 \d \rangle$.
Let us investigate, for which $u$ the properties change.
\small
\begin{verbatim}
LIB "control.lib";
ring r3 = (0,g,m1,m2,L1,L2,d1,d2,k1,k2,u),(d),(c,dp);
poly z1 = 0; poly z2 = u;
module RR =  
             [m1*L1^2*d^2+d1*d+z1, 0, m1*L1*d^2],
             [0, m2*L2^2*d^2+d2*d+z2, m2*L2*d^2];
module R = transpose(RR);
module S = groebner(R);
matrix T = lift(R,S);
genericity(T);
\end{verbatim}
\normalsize

The output of \texttt{genericity} delivers
\small
\begin{verbatim}
[1]:
   u,m2,L2,d1
[2]:
   m1^2*L1^4*u-m1*L1^2*d1*d2+m2*L2^2*d1^2
\end{verbatim}
\normalsize

That is, the generic annihilator of a torsion submodule of the
system subject to constraints $k_1 - m_1 L_1 g =0, k_2 - m_2 L_2 g =u \not=0$
is indeed $\langle \d \rangle$. However, if 
$u = k_2 - m_2 L_2 g = \frac{m_{1}L_{1}^{2}d_{2}-m_{2}L_{2}^{2}d_{1}}{m_{1}^{2}L_{1}^{4}} d_1$,
the non--generic annihilator equals $\langle m_1 L_1^2 \d ^2 + d_1 \d \rangle$. This
illustrates the difference between two components, corresponding to cases 2 and 3. \\

Case 4). $k_1 \not= m_1 L_1 g, k_2 = m_2 L_2 g, P = 0$ and \\

Case 5). $k_1 \not= m_1 L_1 g, k_2 = m_2 L_2 g, P \not= 0$.

The simplified presentation matrix for the transposed module
is symmetric, that is, exchanging $m_1 \leftrightarrow m_2$,
$L_1 \leftrightarrow L_2$, $d_1 \leftrightarrow d_2$ and
$k_1 \leftrightarrow k_2$ simultaneously does not change the matrix. Hence,
we can take the results of case 2 respectively case 3, exchange the
variables and get the results for case 4 respectively case 5. \\

Case 6). $k_1 \not= m_1 L_1 g, k_2 \not= m_2 L_2 g, P = 0$.

Recall the special structure of a polynomial $P$ in (\ref{Peq}).
It is easy to see, that if $P=0$ and one of the two summands of $P$ is zero, 
so does the other. This observation lead us to the first conclusion:

\[
P = 0, \text{ if  } \dfrac{L_2}{L_1} = \dfrac{z'_2}{z'_1} = \dfrac{d'_2}{d'_1}.
\]

Going back to the original variables, it translates into

\begin{equation}
\label{FracEq}
\frac{m_2 L^2_2}{m_1 L^2_1} = \frac{k_2 - m_2 L_2 g}{k_1 - m_1 L_1 g} = \frac{d_2}{d_1}.
\end{equation}

This is especially interesting, since the values, found in \cite{PW} for showing
the non--generic non--controllability, were $m_1 = m_2 = M_0 = 1$, $L_1 = L_2 = 1$, $d_1 = d_2 = 1$ and $k_1 = k_2 = k$. As we can see, it suffices to set $m_1 = m_2$, $d_1 = d_2$, $k_1 = k_2 \not= m_2 L_2 g$ and $L_1 = L_2$ for illustrating this phenomenon.

Let us denote by a parameter $t$ the value of the fractions in \ref{FracEq}. Then,
\[
 d_2 = t \cdot d_1, \; k_2 = t \cdot k_1 + (m_2 L_2 - t \cdot m_1 L_1)g, \;  m_2 L_2^2 = t \cdot m_1 L_1^2
\]

We do the substitutions for $d_2$ and $k_2$. As a preprocessing before Gr\"obner bases, we can manipulate the generators. Consider the last generator of a transposed module, that is, the last column of the transposed presentation matrix $(m_1 L_1 \d^2, m_2 L_2 \d^2)^T$. By multiplying the column with $L_2$, we can simplify it subject to the substitution to the column $(L_2 \d^2, t L_1 \d^2)^T$. The second generator of the module becomes then $(0,t\cdot(m_1 L_1^2 \d^2+d_1 \d+z_1))^T$, from which we cancel the parameter $t$ out. With the following code we perform the controllability and the autonomy analysis for this particular case.

\small
\begin{verbatim}
ring r6 = (0,g,t,m1,L1,L2,d1,k1),(d),(c,dp);
poly z1 = k1 - m1*L1*g;
module RR =  
          [m1*L1^2*d^2+d1*d+z1, 0, L2*d^2],
          [0, m1*L1^2*d^2+d1*d+z1, t*L1*d^2];
module R = transpose(RR);
print(R);
list LC = canonize(control(R));
list LA = canonize(autonom(R));
\end{verbatim}
\normalsize

We conclude that this system is neither controllable nor autonomous.
In particular, the annihilator of the torsion submodule is the ideal
$\langle m_{1}L_{1}^{2}\d^{2}+d_{1}\d+k_{1}-m_{1}L_{1}g \rangle$. Note that in view
of the equation (\ref{FracEq}), we obtain the equivalent symmetric annihilator $\langle m_{2}L_{2}^{2}\d^{2}+d_{2}\d+k_{2}-m_{2}L_{2}g \rangle$ by e.g. multiplying the previous annihilator with the constant $t$.\\

Now let us assume, that $P=0$ but neither of its summands vanishes.
The polynomial $P$ is quadratic with respect to any of the variables $m_1, m_2, k_1, k_2, d_1, d_2$
 and is quartic with respect to $L_1$ and $L_2$. 
Let us fix one of the variables $m_1, m_2, k_1, k_2, d_1, d_2$. Consider the rest of variables as parameters and compute the discriminant of a corresponding quadratic equation. Since the involved variables might have only positive real values, we obtain a condition on the discriminant of a quadratic equation. If we fix $m_1, k_1$ or $d_1$, we get $\frac{d_2^2}{4m_2 L_2^2} \geq k_2 - m_2 L_2 g$. For fixed $m_2, k_2$ or $d_2$, we obtain, either by a direct computation or via the symmetry, that $\frac{d_1^2}{4m_1 L_1^2} \geq k_1-m_1 L_1 g$. Notably, both inequalities
cannot become equalities simultaneously.

Provided $\frac{d_2^2}{4m_2 L_2^2} \geq k_2 - m_2 L_2 g$, the explicit solution with respect to, say, $d_1$ gives the following expression (recall, we use the short notation $z_i = k_i - m_i L_i g$):
\[
d_1 = \dfrac{d_2(m_2 L_2^2 z_1 + m_1 L_1^2 z_2) \pm (m_2 L_2^2 z_1 -  m_1 L_1^2 z_2) \sqrt{d_2^2 - 4m_2 L_2^2 z_2}}{2 m_2 L_2^2 z_2}
\]

Each root corresponds to a separate system. Substituting the roots into our system, we obtain, that as in all previous cases, it is neither controllable nor autonomous. The annihilators of torsion submodules are then $\langle 2 m_2 L_2^2 \d + d_2 \pm \sqrt{d_2^2 - 4m_2 L_2^2 z_2} \rangle$. The annihilators with respect to $m_1,d_1,L_1,z_1$ we obtain by the symmetry.

Finally, we summarize the obtained results.

\begin{proposition}
The complete stratification of the obstructions to genericity for the generically controllable system with the essential friction is obtained. All the components
of the stratification correspond to non--controllable and non--autonomous systems, whose torsion submodules are annihilated by one of the ideals (for $i=1,2$)
\[
\langle m_{i}L_{i}^{2}\d^2+d_{i}\d \rangle, \;\langle m_{i}L_{i}^{2}\d^{2}+d_{i}\d+k_{i}-m_{i}L_{i}g \rangle, \; \langle \d \rangle,
\]

and $\langle 2 m_i L_i^2 \d + d_i \pm \sqrt{d_i^2 - 4m_i L_i^2 (k_i - m_i L_i g)} \rangle$, provided 
 $k_i \leq \frac{d_i^2}{4m_i L_i^2} + m_i L_i g$.
\end{proposition}


\section{Conclusion and Future Work}

We have investigated the parameter-dependence of structural properties 
(such as torsion-freeness) of modules 
over certain algebras over ${\K}(p_1,\ldots,p_t)$,
where $\K$ is a ground field and $p_i$ are parameters.
The central idea is to keep track of all polynomial expressions in the
$p_i$ that occur as denominators during Gr\"obner basis computation. 
These problems have practical applications in control theory as outlined in
the Introduction. We have shown several nontrivial phenomena
that arise with these questions in terms of illustrative worked 
examples. Our goal for the future is to extend this approach to 
the study of more general parametric module 
properties, leading to the implementation of systematic procedures for such problems. \\

In particular, one is interested in working with parameters, on
which the involved operators act nontrivially. That
is, the parameters may correspond to \\
($q$--)differentiable and/or ($q$--)shiftable
functions. Then, the field $\K(p_1,\ldots,p_t)$ must be a differential
and/or a difference field. The obstructions to genericity are then presented
as systems of differential--difference algebraic equations (DDAE) instead of
just algebraic equations treated in this article. Though the main principles
remain the same, there is a strong need for specialized techniques and 
systematic computer--algebraic support for both theoretical and implementational parts
of the further research in this area. The case of differentiable parameters was treated in the
articles \cite{PQ1997,PQ2000}, the software package \textsc{OreModules} \cite{OreModules}
seems to be able to provide computational support for this case. \\

Yet another important direction of investigation is the analysis of numerical phenomena,
namely inexact computations with parameters defined as floating point numbers or as certain inequalities. The generalization of our approach to these domains seems to be possible with the help of e.g. cylindrical algebraic decomposition techniques. Alternatively, one may first
obtain an exact symbolic solution to parametric problem, say, in form of the complete
stratification, and postprocess it with numerical or symbolical--numerical tools. \\

\textbf{Acknowledgements}. The authors are grateful to A.~Quadrat and J.-F.~Pommaret for discussions on theoretical methods as well as particular applications. The first author is grateful
to the SFB project F1301 of the Austrian FWF for partial financial support.


\bibliographystyle{abbrv}                       %
\bibliography{database}                             %

\end{document}